%% file: ms.tex
\documentclass{IEEEtran}
\input{nando_commands}

\usepackage{ifpdf}
\usepackage{cite}

\ifCLASSINFOpdf
  \usepackage[pdftex]{graphicx}
\else
  \usepackage[dvips]{graphicx}
\fi

\usepackage{amsmath}
\usepackage{algorithmic}
\usepackage{array}

\ifCLASSOPTIONcompsoc
 \usepackage[caption=false,font=normalsize,labelfont=sf,textfont=sf]{subfig}
\else
  \usepackage[caption=false,font=footnotesize]{subfig}
\fi
\usepackage{url}
\usepackage[export]{adjustbox}

\begin{document}

\title{Privacy-Preserving Obfuscation for \\Distributed Power Systems}

%% To specify the authors when (number of affiliations <= 2)
\author{
\IEEEauthorblockN{Terrence W.K. Mak\IEEEauthorrefmark{1}, Ferdinando Fioretto\IEEEauthorrefmark{1}\IEEEauthorrefmark{2},
 and Pascal Van~Hentenryck\IEEEauthorrefmark{1}}
 
\IEEEauthorblockA{
\IEEEauthorrefmark{1} Georgia Institute of Technology, Atlanta, GA, USA\\
\IEEEauthorrefmark{2} Syracuse University, New York, NY, USA \\
wmak@gatech.edu, ffiorett@syr.edu, pvh@isye.gatech.edu}
}

\maketitle\sloppy\allowdisplaybreaks

\begin{abstract}
  This paper considers the problem of releasing privacy-preserving  load data of a decentralized operated power system. 
  The paper focuses on data used to solve Optimal Power Flow (OPF) problems and proposes a distributed algorithm that complies with the notion of \emph{Differential Privacy}, a strong privacy framework used to bound the risk of re-identification.
  The problem is challenging since the application of traditional differential privacy mechanisms to the load data fundamentally changes the nature of the underlying optimization problem and often leads to severe feasibility issues.
  The proposed \emph{differentially private distributed algorithm} is based on the Alternating Direction Method of Multipliers (ADMM) and guarantees that the released privacy-preserving data 
  retains high fidelity and satisfies the AC power flow constraints. 
  Experimental results on a variety of OPF benchmarks demonstrate 
  the effectiveness of the approach. 
\end{abstract}


\section{Introduction}

The availability of test cases representing high-fidelity power system networks is essential to foster research in several important power optimization problems, including optimal power flow (OPF), unit commitment, and transmission planning. 
However, the release of such datasets poses significant privacy risks. For instance, revealing the electrical load of a customer 
may disclose sensitive business activities and  manufacturing processes, causing significant economic loss. 
Indirectly, it may also reveal how transmission operators operate their networks, raising security issues \cite{Fioretto:ijcai-19}. 

\emph{Differential Privacy (DP)} \cite{Dwork:06} is a privacy framework that has been shown effective in protecting sensitive information during a data release process. 
It prevents the disclosure of sensitive information 
by introducing carefully calibrated noise to the result of a computation. 
While DP algorithms could be used \emph{directly} to generate privacy-preserving power system data, they face significant challenges when the released data is required to preserve domain specific properties, such as preserving the optimal cost and the feasibility of an AC Optimal Power Flow (AC-OPF) problem. 
Naive noise addition can drastically degrade the fidelity to the original problem of interest and introduce severe feasibility issues, as shown in~\cite{mak19privacy,Fioretto:ijcai-19,fioretto19differential}. Fig.~\ref{fig:lap_err}, reported from \cite{mak19privacy}, emphasizes these results. It shows the average load distance (as $L_1$) between the original and the privacy-preserving loads for a set of 29 networks, at varying obfuscation parameter $\alpha$. The percentages of instances with a feasible AC-OPF solution are shown above the bars. 

% ------------------------------------ %
\begin{wrapfigure}{r}{100pt} 
\vspace{-10pt} 
\centering %
\hspace{-10pt}
\includegraphics[width=\linewidth]{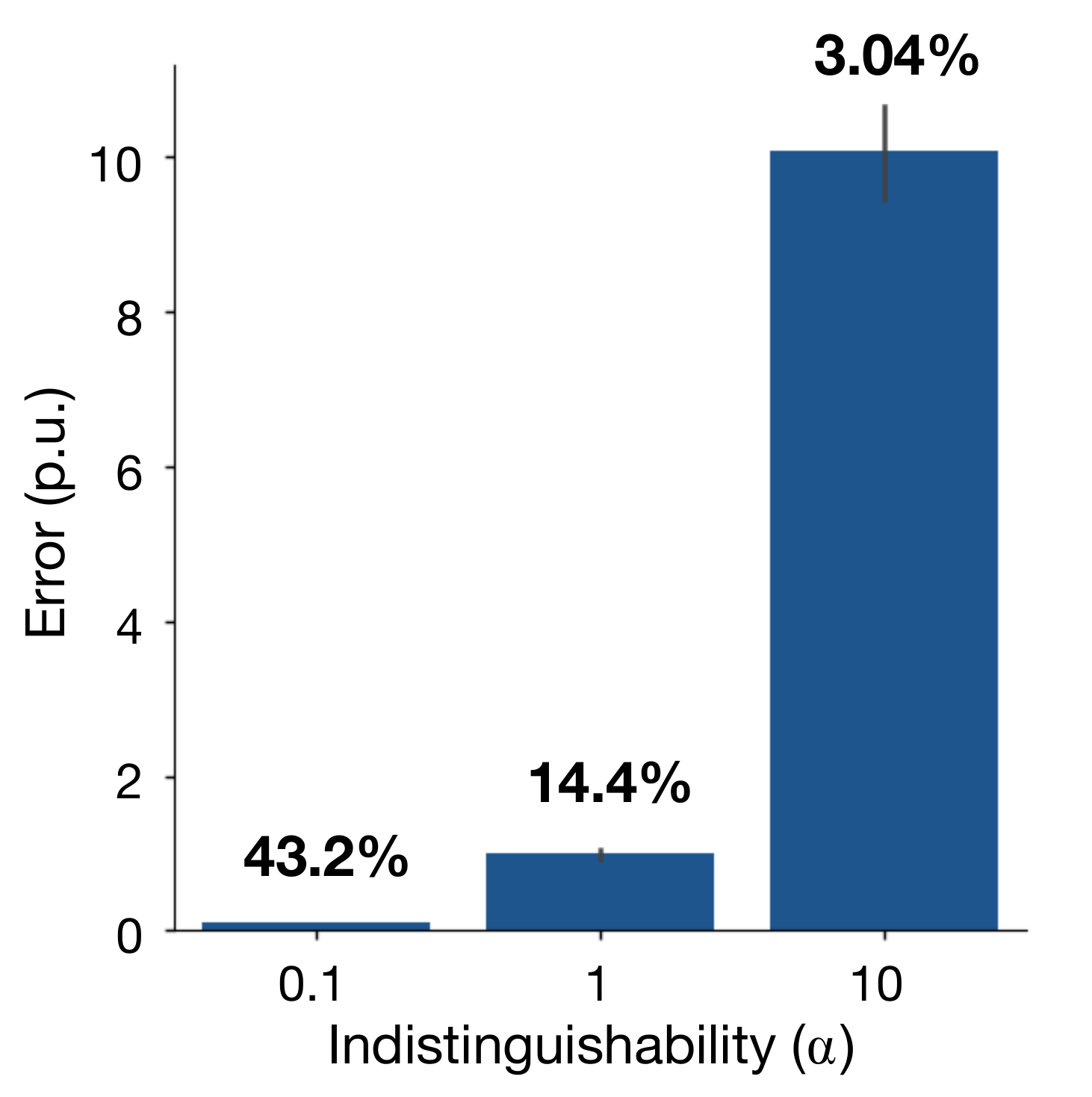}\\ %
\caption{Average $L_1$ error 
and percentages of feasible AC-OPF instances.} %
\label{fig:lap_err} %   
\vspace{-10pt} 
\end{wrapfigure} %
% ------------------------------------  %
Interestingly, a recent body of work has shown that it is possible to release AC-feasible obfuscated load data that also satisfies the notion of differential privacy \cite{Fioretto:18b, fioretto19differential, mak19privacy}.
Despite the soundness and effectiveness of such data release techniques, these methods rely on the presence of a trusted data curator that can collect sensitive loads from all the system participants.
However, this is impractical in very large systems with distributed loads and generators (e.g., multiple microgrids). 
Even if the power system is operated centrally, it is typically owned and controlled by various parties, e.g., load customers, transmission system operators (TSO), distribution system operators (DSO), and generation companies. 
These parties operate with specific customer and legal agreements, which render the transmission of proprietary data to a centralized server infeasible.

To overcome these limitations, this paper introduces the \emph{Privacy-preserving Decentralized OPF (PD-OPF)}, 
a novel decentralized and privacy-preserving framework that allows multiple power system parties to release their data privately without relying on a trusted data curator. 
Crucially, the framework guarantees that the released data produces a feasible AC power flow, and that its OPF cost is close to that of the original OPF. 
The heart of the mechanism is a distributed optimization procedure that relies on the \emph{Alternating Direction Method of Multipliers (ADMM)} to redistribute the noise introduce by traditional DP algorithms to satisfy the desired properties. 

While the paper focuses on preserving the privacy of individual loads, the framework is general and can be used to protect other sensitive quantities (e.g., generator capabilities).

\smallskip\noindent\textbf{Contributions}
The key contributions of this work are as follows: 
(1) It introduces DP-OPF, a novel, distributed, differentially private mechanism that relies on ADMM to obfuscate the individual loads while ensuring AC-OPF feasibility on the obfuscated data. 
(2) DP-OPF satisfies the notion of $\epsilon$-local differential privacy, providing a strong privacy guarantees.
(3) Experimental results on a large collection of OPF benchmarks illustrate that the proposed approach finds high-quality
AC-feasible solutions, and that the results are comparable to those obtained with a centralized version with a data curator.
% obfuscation quality, and comparable to solutions of a centralized obfuscation mechanism.

%%%%%%%%%%%%%%%%%%%%%%%%%%%%%%%%
\section{Related Work}
\label{sec:related work}
%%%%%%%%%%%%%%%%%%%%%%%%%%%%%%%%

There is a rich literature on theoretical results of DP (see for instance \cite{dwork:13} and \cite{vadhan:17}).
The literature on DP applied to power systems includes considerably fewer efforts.  {\'A}cs and Castelluccia \cite{acs:11} exploit a direct application of the Laplace mechanism to hide user participation in smart meter data sets, achieving $\epsilon$-DP.  Zhao et al.~\cite{zhao:14} study a DP schema that exploits the ability of households to charge and discharge a battery to hide the real energy consumption of their appliances. Liao et al.~\cite{liao:17} introduce Di-PriDA, a privacy-preserving mechanism for appliance-level peak-time load balancing control in the smart grid, aimed at masking the consumption of top-k appliances of a household. Finally, Zhou et al.~\cite{zhou2019} introduce the notion of monotonicity of DC-OPF operator, which requires that monotonic changes in the network loads induce monotonic changes in the DC-OPF objective cost. This enables a characterization of the network, which is useful to preserve the privacy of \emph{monotonic networks}.

There are also related work on privacy-preserving implementations of the ADMM algorithm. Zhang et al.~\cite{zhang2016dynamic} proposed a  version of the ADMM algorithm for privacy preserving empirical risk minimization problems, a class of convex problems used for regression and classification tasks. 
Huang et al.~\cite{huang2019dp} proposed an approach that combines an approximate augmented Lagrangian function with time-varying Gaussian noise for general objective functions. 
Finally, Ding et al.~\cite{ding2019optimal} proposed P-ADMM, to provide guarantees within a \emph{relaxed} model of differential privacy (called zero-concentrated DP).

The privacy-preserving distributed learning literature focuses almost entirely on problems whose objective functions are smooth and strongly convex. 
Additionally, most approaches suffer one shortcoming: The privacy loss being provided as a guarantee 
is a function of the iteration counts of the algorithm, which can be huge if 
a large number of iterations is required to converge to a feasible solution.
In contrast, this work provides bounded privacy loss irrespective of the number of iterations. 
It also ensures that the privacy-preserving data is AC-OPF feasible and that the solution cost stays close to the original ones.

%%%%%%%%%%%%%%%%%%%%%%%%%%%%%%%%
\section{Preliminaries}
\label{sec:background}
%%%%%%%%%%%%%%%%%%%%%%%%%%%%%%%%

\subsection{Optimal Power Flow}
\label{sec:opf}

\emph{Optimal Power Flow (OPF)} is the problem of determining the most economic generator dispatch to meet the load demands in a power network.  
A power network $\bm{\cN}$ can be viewed as a graph $(N, E)$ where the set of buses $N = [n]$ represents the nodes and the set of lines and transformers $E \subseteq \{(i,j) \in N\times N\}$ represents the directed arcs.  
The paper denotes with $G$ and $L$ as for the set of generators and loads in the network, and uses $E^R$ to indicate the set of arcs, but in the reverse direction. 
The AC-OPF problem ($P_{\text{OPF}}$) is specified in Model \ref{model:ac_opf}, where $I, V, Y,$ and $S$ denote the complex quantities for current, voltage, admittance, and power, respectively. 

The model takes as input the power network $\bm{\cN}$ and returns the optimal generator dispatch costs (with ties broken arbitrarily). 
The objective function ${\cO}(\bm{S^g})$ captures the cost of the generator dispatch, with $\bm{S^g} = \langle S^g_1, \ldots, S^g_n\rangle$ denoting the vector of generator dispatch values.
Constraint \eqref{eq:ac_0} sets the reference angle to zero for the slack bus $s \in N$ to eliminate numerical symmetries.  Constraints \eqref{eq:ac_1} and \eqref{eq:ac_2} capture the voltage and phase angle difference bounds. Constraints \eqref{eq:ac_3} and \eqref{eq:ac_4} enforce the generator output and line flow limits.  Finally, Constraints \eqref{eq:ac_5} capture Kirchhoff's Current Law and Constraints \eqref{eq:ac_6} capture Ohm's Law.
The solution set satisfying Constraints \eqref{eq:ac_0} to \eqref{eq:ac_6} for a given set of load demands $\bm{S}^d = \langle S_1^d, \ldots, S_n^d\rangle$ is denoted by $\sC_{PF}(\bm{S}^d)$.  
Table \ref{tab:notation} summarizes the common notations used throughout the paper.

%--------------------------------------------------------------------
\begin{table*}[t]
\caption{Common notation used in the paper.\label{tab:notation}}
\begin{center}
	% \resizebox{0.99\linewidth}{!}
	{
	\begin{tabular}{l l  l l}
    \toprule
		{$\bm{\cN}$}  			& Power network								&  {$\epsilon$} & Privacy budget\\
		{$\bm{S}^g$}  			& Vector of power generator dispatch 					&  {$\alpha$} & Indistinguishability value \\
		{$\bm{S}^d$}  			& Vector of load demands					 		&  {$\beta$} & Faithfulness value/parameter\\
		{$P_{\text{OPF}}$}		& Function solving AC-OPF, with input $\bm{S}^d$ and output $\bm{S}^g$ 							&  {$\cO^*$} & The optimal costs of the original problem \\
		{$\sC_{\text{PF}}$}		& The set of feasible AC power flow for $P_{\text{OPF}}$ 	& {$\bm{x}$} & A vector of variables/values \\
		{$\cM_x$}				& A mechanism of $x$							& {$x^l$, $x^u$} & Upper and lower bounds of quantity $x$\\
		{$P_x$}				& The optimization problem for the accuracy phase of $\cM_x$	&{$\Re(\cdot), \Im(\cdot)$} & Real / imaginary component of a complex number\\
		{$Y^*, I^*, V^*$}		& Conjugate of admittance matrix $Y$, current $I$, and voltage $V$ & $c_2, c_1, c_0$& Cost function coefficients \\
		{$\lambda^d, \lambda^g, \lambda^V, \lambda^S$} & Lagrange multiplier for load, generation, voltage, and power flow & $\rho$& ADMM penalty parameter\\
		\bottomrule
  	\end{tabular}
  	}
	\end{center}
  	
\end{table*}
%--------------------------------------------------------------------

\begin{model}[!t]
	{\footnotesize
	\caption{AC Optimal Power Flow: $P_{\text{OPF}}$ }
	\label{model:ac_opf}
	\vspace{-6pt}
	\begin{align}
		\mbox{\bf variables:} \;\;
		& S^g_i, \; \forall i \in G; \;\;
		V_i, \; \forall i\in N; \;\;
		  S_{ij}, 	 \; \forall(i,j)\in E \cup E^R \nonumber \\
		\mbox{\bf minimize:} \;\;
		& {\cO}(\bm{S^g}) = \sum_{i \in N} {c}_{2i} (\Re(S^g_i))^2 + {c}_{1i}\Re(S^g_i) + {c}_{0i} \label{ac_obj} \\
		\mbox{\bf subject to:} \;\; 
		& \angle V_{s} = 0, \;\; \exists s \in N \label{eq:ac_0} \\
		& {v}^l_i \leq |V_i| \leq {v}^u_i 		\;\; \forall i \in N \label{eq:ac_1} \\
		& -{\theta}^\Delta_{ij} \leq \angle (V_i V^*_j) \leq {\theta}^\Delta_{ij} \;\; \forall (i,j) \in E  \label{eq:ac_2}  \\
		& {S}^{gl}_i \leq S^g_i \leq {S}^{gu}_i \;\; \forall i \in G \subseteq N \label{eq:ac_3}  \\
		& |S_{ij}| \leq {s}^u_{ij} 					\;\; \forall (i,j) \in E \cup E^R \label{eq:ac_4}  \\
		& S^g_i - {S}^d_i = \textstyle\sum_{(i,j)\in E \cup E^R} S_{ij} \;\; \forall i\in N \label{eq:ac_5}  \\ 
		& S_{ij} = {Y}^*_{ij} |V_i|^2 - {Y}^*_{ij} V_i V^*_j 			 \;\; \forall (i,j)\in E \cup E^R \label{eq:ac_6}
	\end{align}
	}
	\vspace{-12pt}
\end{model}

%%%%%%%%%%%%%%%%%%%%%%%%%%
\subsection{Alternating Direction of Multipliers Method (ADMM)}
\label{sec:admm}
%%%%%%%%%%%%%%%%%%%%%%%%%%
ADMM is a widely used distributed procedure solving optimization problems with coupling constraints.
Consider an optimization problem of the following form:
\begin{align}
\min_{x \in \cX, z \in \cZ}\quad& f(\bx) + g(\bz) \nonumber\\
\mbox{ s.t. } \quad& A\bx + B\bz = \bc, \label{math:admm}
\end{align}
where $\cX \subseteq \RR^n$ and $\cZ \subseteq \RR^m$ are two disjoint sets, $\bx \in \RR^n$ and $\bz \in \RR^m$ denote variable vectors owned by two distinct groups of agents, and $A \bx + B \bz = \bc$ describes the set of coupling constraints \emph{between} the two groups of agents with $A \in \mathbb{R}^{\ell \times n}$, $B \in \mathbb{R}^{\ell \times m}$, and $\bc \in \mathbb{R}^{\ell}$. 
The functions $f$ and $g$ denote the objectives over $\bx$ and $\bz$, respectively, and are commonly assumed to be convex.
The augmented Lagrange function $L(\bx, \bz, \bm{\lambda})$ of (\ref{math:admm}) is:
\begin{align}\label{math:lag}
f(\bx) + g(\bz)  + \bm{\lambda} (A\bx + B\bz - \bc) + \frac{\rho}{2} | A\bx + B\bz - \bc|^2
\end{align}
where $\bm{\lambda} \in \RR^\ell$ is a vector of \emph{Lagrangian multipliers} and $\rho > 0$ is a penalty term. 
The vector of Lagrangian multipliers are the dual variables associated with the coupling constraints $A \bx + B \bz = \bc$. 
Given a solution tuple $(\bx^i, \bz^i, \bm{\lambda}^i)$ at iteration $i$, 
ADMM~\cite{boyd2011distributed} proceeds to the next iteration $i + 1$ computing $(\bx^{i+1}, \bz^{i+1}, \bm{\lambda}^{i+1})$
as follows, in three sequential steps:
\begin{align}
\bx^{i+1} &= \argmin_{\bx \in \cX} L(\bx, \bz^i, \bm{\lambda}^i) \label{math:admm_stepx} \\
\bz^{i+1} &= \argmin_{\bz \in \cZ} L(\bx^{i+1}, \bz, \bm{\lambda}^i) \label{math:admm_stepy}\\
\bm{\lambda}^{i+1} &= \bm{\lambda}^{i} + \rho (A\bx^{i+1} + B\bz^{i+1} - \bc). \label{math:admm_lambda}
\end{align}
The algorithm terminates when a desired termination condition is reached (e.g., an iteration limit or a convergence factor).
The quality of the solution at iteration $i$ can be measured by the primal infeasibility (residue) vector~\cite{mhanna19adaptive}
\begin{align}
\bm{r}_{p}^i = A\bx^i + B\bz^i - \bc,
\end{align}
indicating the distance to a primal feasible solution, and 
the dual infeasibility (residue) vector~\cite{mhanna19adaptive}
\begin{align}
\bm{r}_{d}^i = \rho A^{T} B (\bz^i - \bz^{i - 1}),
\end{align}
indicating the distance from the previous local minima.
When both infeasibility vectors are zero, ADMM
converges to a (local) optimal and feasible solution.

%%%%%%%%%%%%%%%%%%%%%%%%%%
\subsection{Differential Privacy Notions}
\label{sec:dp}
%%%%%%%%%%%%%%%%%%%%%%%%%%

The need for data privacy emerges in two main contexts: the \emph{global} context, as in when institutions release datasets containing information of several users or answer queries on such datasets (e.g., US Census queries \cite{abowd2018us,Fioretto:cp-19}), and the \emph{local} context, as in when individuals disclose their personal data to some data curator (e.g., Google Chrome data collection process \cite{fanti:16}). 
In both contexts, privacy is achieved through a randomizer ${\cal M}$ adding noise to the data before releasing. 

\emph{Differential privacy} \cite{Dwork:06} (DP) is an algorithmic property that characterizes and bounds the privacy loss of an individual when its data participates into a computation. 
It has originally been proposed in the global privacy context and, informally, ensures that an adversary would not be able to reliably infer whether or not a particular individual participates in the dataset, even with unbounded computational power and access to every other entry of the dataset. 
The setting adopted in this work studies the local privacy context (LDP) \cite{LDP_dwork2006}, in which each load customer $i$ holds a datum, $S_i^d \in \CC$, describing the complex load consumption of the bus $i \in N$. While the standard local differential privacy framework is concerned with protecting the participation of an individual into a dataset, in a power system, the individual identity is not a sensitive information: It is a public knowledge that each bus may connect to a demand.
The sensitive information is represented by the load magnitude. 
To accommodate such notion of privacy risk, the paper uses the definition of \emph{generalized differential privacy for metric spaces} \cite{chatzikokolakis2013broadening} and adapts it to the local differential privacy context. 
Without loss of generality, we focus on Lebesgue spaces $\bm{L}^1$, and in particular, consider the complex space $\CC$ equipped with norm 1. 
For a given value $\alpha > 0$, a randomized mechanism ${\cal M}$ is $\epsilon$-LDP for $\alpha$ distances (a.k.a.~local $\alpha$-indistinguishable), if
for all $x$ and $x' \in \CC$ s.t. $\| x - x'\|_1 \leq \alpha$, and for any output response $o \in \CC$:
\begin{equation}
	\label{eq:LDP}
	\Pr[ {\cal M}(x) = o] \leq e^\epsilon \Pr [{\cal M}(x') = o].
\end{equation}

Informally, the LDP definition adopted ensures that an attacker obtaining access to a privacy-preserving load value cannot detect, with high probability, 
the distance between the privacy-preserving value and its original value. 
The level of \emph{privacy} is controlled by the privacy loss parameter $\epsilon \geq 0$, with small values denoting
strong privacy. 
The level of \emph{indistinguishability} is controlled by the parameter $\alpha > 0$. 
The above definition allows us to \emph{obfuscate} load values that are close to one another while retaining the distinction between those that are far apart. 
%Local $\epsilon$-differentially private algorithms are $\epsilon$-differentially private \cite{dwork:13} in the global context. Similarly, local $\alpha$-indistinguishable algorithms are also $\alpha$-indistinguishable in the global context. 
Local Differential Privacy (LDP), including its extension for generic metric spaces, satisfies several important properties. 
In particular, it is immune to post-processing as defined in the following theorem. 
\begin{theorem}[Post-Processing Immunity]\cite{dwork:13} 
\label{th:postprocessing} 
Let $\cM$ be an $\epsilon$-(local) differentially private mechanism and $g$ be 
an arbitrary mapping from the set of possible output 
sequences to an arbitrary set. Then $g \circ \cM$ is $\epsilon$-(local) differentially private.
\end{theorem}

\section{Decentralized Load Obfuscation}
\label{sec:OLI}
The decentralized load obfuscation problem is the problem of coordinating the release privacy-preserving load data in a power system owned and controlled by multiple parties. We consider a set of agents, each coordinating some power system component, e.g., loads, generators, buses, or power lines.
The goal of the problem is to release load data, which is controlled by the load agents.

The problem has three desiderata. 
(1) It requires obfuscation of the loads up to some amount $\alpha>0$. 
(2) It requires that the AC-OPF objective induced by the obfuscated loads is close to that attained using the original data.
(3) It requires its agents to coordinate the data release process using a decentralized and confined communication process. 

Formally, the decentralized load obfuscation problem finds the active and reactive, privacy-preserving load values $\hat{S}^d_i$ for each load agent $i \in N$ that satisfy the following criteria:
\begin{enumerate}

\item \textit{\textbf{Privacy}}: 
	The original load $S^d_i$ and its privacy-preserving counterpart $\hat{S}_i^d$ are local $\alpha$-indistinguishable, for every load $i \in N$. 

\item \textit{\textbf{Fidelity}}: \label{cond:2}	
	For every generator $i$, the optimal AC-OPF dispatch cost 
	$\cO(\hat{S}^g_i)$
	obtained by using the obfuscated loads $\hat{S}_i^d$ is required to be close to the original AC-OPF dispatch cost 
	$\cO({S}^g_i)$ up to a user-defined factor $\beta > 0$:
	\begin{equation}
	 |\cO(\hat{S}^g_i) -  \cO(S^g_i) | 
	 \leq \beta \cO(S^g_i) \;\; \forall i \in N.\nonumber
	\end{equation}
\end{enumerate}
Finally, it requires the \textit{\textbf{computation mechanism}} to be performed in a decentralized fashion.
In the following, we denote with $\cO^*_i$ as for the original optimal generation costs $\cO(S^g_i)$, which are assumed to be publicly known \cite{Fioretto:18b} (e.g., from the market information).

\section{The PD-OPF Mechanism}
This section introduces the \emph{Privacy-preserving Distributed OPF (PD-OPF)} mechanism to solve the decentralized load obfuscation problem.
PD-OPF agents operate in two phases:
\begin{enumerate}
	\item \textbf{Privacy Phase} 
	During the first phase, each load agent $i \in N$ applies a LDP protocol to obtain an $\alpha$-local obfuscated version $\tilde{S}^d_i$ of its original load $ S^d_i$. 
	This process is executed independently and autonomously by each load agent in the system. 
	
	\item \textbf{Fidelity Phase} 
	In the second phase, the agents coordinate a distributed process to adjust the private load values $\tilde{S}_i^d$, to new values 
	$\hat{\bS}^d$ that achieve the fidelity goal, while deviating as little as possible from the local $\alpha$-obfuscated loads $\tilde{S}_i^d$.
\end{enumerate}
The next sections describe in details the PD-OPF phases.

\subsection{Privacy Phase}
In the privacy phase, each (load) agent $i$ perturbs its load data $S^d_i$, independently from other agents, so to generate an $\alpha$-local indistinguishable load $\tilde{S}_i^d$. 
To do so, the agents use a version of the \emph{Laplace Mechanism}, a method used to guarantee an $\epsilon$-LDP private responses to numeric functions \cite{dwork:13}. 
The Laplace distribution with $0$ mean and scale $b$, denoted by $\Lap(\xi)$, has a probability density function $\Lap(x|\xi) = \frac{1}{2\xi} e^{-\frac{|x|}{\xi}}$. 
Let $\Lap(\xi)$ to be the Laplace distribution with parameter $\xi$, $f$ a numeric function that maps datasets to $\RR$, and $z$ to be a random variable drawn from
$\textrm{\Lap}\left(\frac{\alpha}{\epsilon}\right)$. 
%Therein, $\Delta_f$, called the \emph{sensitivity} of $f$, denotes the maximal change in the output $f(x)$ for any input $x \in \RR^n$: 
%$\Delta_f = \text{sup}_{x,x' \in \RR^n} \|f(x) - f(x')\|_1$. 
The Laplace mechanism for local differential privacy for $\alpha$ distances is defined as follows:
\begin{theorem}[Laplace Mechanism]
\label{th:m_lap}
The Laplace mechanism that outputs $f(x) + z$ achieves $\alpha$-local indistinguishability.
\end{theorem}

Since the load data is represented in the complex form, agents use the Polar Laplacian mechanism~\cite{andres2013geo,mak19privacy}, which is a generalization of the Laplace mechanism to Euclidean spaces. 
%Given a point $(r, \theta)$ in polar coordinates.
%The probability density function $\PLap(r, \theta | \xi)$ is given by
%$\frac{\xi^2}{2\pi} r\,e^{-\xi\,r}$. 
The mechanism %with parameter $\xi=\alpha/\epsilon$ 
satisfies $\alpha$-local obfuscation \cite{chatzikokolakis2013broadening,mak19privacy}.
For simplicity, the paper refers to the the Laplace mechanism as for the Polar Laplace mechanism.

\subsection{Fidelity Phase}
While simply adding Laplace noise to each load satisfies 
local $\alpha$-indistinguishability, the resulting power system data may no longer be AC feasible, nor it may induce a similar optimal dispatch costs. 
To find a set of loads $\hat{\bS}^d$ that satisfy the fidelity criteria, a post-processing step that uses a bi-level program $P_{BL}$ can be formulated as follows~\cite{mak19privacy}:
\begin{align}
	P_{BL} = \min	& \;\; \| {\hat{\bm{S}}}^d - {\tilde{\bm{S}}}^d \|^2
			\label{obj:oli_upper} \\
	\text{s.t.:}
	& \quad \big| {\cO}(\bm{S}^g) - \cO^* \big| 
		\leq \beta  \cO^*
		\label{c:oli_fait}  \\
	& \quad \bm{S}^g = P_{\text{OPF}}(\bm{\hat{S}}^d)
			\label{c:oli_set} .
\end{align}
The upper level objective Eq.~\eqref{obj:oli_upper}
minimizes the L2 distance between the noisy loads 
${\tilde{\bm{S}}}^d$ and the (post-processed) load 
variables $\hat{\bm{S}}^d$.
Constraint \eqref{c:oli_set} captures the AC-OPF requirement. 
It computes an AC optimal generator dispatch $\bm{S}^g$ for the
post-processed loads $\hat{\bm{S}}^d$.
Finally, Constraint \eqref{c:oli_fait} requires the
generator dispatch to satisfy the fidelity goal. 

\def\RBL{\text{RBL}}
Solving bilevel programs is challenging computationally, being strongly NP-Hard \cite{sinha2018review}. To address the underlying computational challenge, an efficient relaxation of problem $P_{BL}$ can be provided as in \cite{mak19privacy}:
\begin{align}
	P_{\RBL} = \min	& \;\; \| {\hat{\bm{S}}}^d - {\tilde{\bm{S}}}^d \|^2 \label{obj:relax} \\
	\text{s.t.:}
	& \quad \big| {\cO}(\bm{S}^g) - \cO^* \big|  \leq \beta  \cO^* \label{relax:fait}  \\
	& \quad \mbox{AC Power Flow: } \eqref{eq:ac_0} - \eqref{eq:ac_6} \label{relax:set} .
\end{align}
It relaxes the optimality requirement Eq.~\eqref{c:oli_set} and only requires AC feasibility (Eq.~\eqref{relax:set}). The mechanism restores feasibility of the loads and ensures the existence of a dispatch whose cost is close to the optimal one.

\subsection{Decentralized Fidelity Phase}
To coordinate the resolution of problem $P_{\RBL}$ in a decentralized fashion, the problem is expressed into the format of Eq. (\ref{math:admm}) and solved using an ADMM protocol. 
The ADMM mechanism used follows the component-based dual decomposition framework~\cite{mhanna18component,mhanna19adaptive} 
and models each power system component as an individual agent.
The framework considers four types of agents: 
load demand agents $\cD$, generator agents $\cG$, line agents $\cL$,
and bus agents $\cB$. %, with state variables independent of each other.
Figure~\ref{fig:framework} illustrates the ADMM communication scheme of adopted by each agent ($i \in N$, if it is a bus, load, or generator agent), or ($(i,j)\in E$, if it is a line agent). 

It is summarized in the following three steps. At each iteration:
\begin{enumerate}
\item 
Load, generator, and line agents compute their individual \emph{consensus variables}, respectively, 
$S^{d(D)}_i$, for load agent $i$, 
$S^{g(G)}_i$, for generator agent $i$, 
and $S^{(L)}_{ij}, V^{(L)}_{ij}$ (and $S^{(L)}_{ji}, V^{(L)}_{ji}$ for the reverse direction) for line agent $(ij)$. 
Collectively, they form a \emph{consensus vector} 
$\bx = \langle S^{d(D)}_i, S^{g(G)}_i, S^{(L)}_{ij}, V^{(L)}_{ij}, S^{(L)}_{ji}, V^{(L)}_{ji} \rangle$ (see Eq.~\eqref{math:admm_stepx}), 
which is sent to their connecting bus agents. 

\item 
Upon receiving its neighboring load, generator, and line consensus variables, bus agent $i$ computes the response value $\bz = \langle S^{d(B)}_i, S^{g(B)}_i, S^{(B)}_{ij}, V^{(B)}_{i} \rangle$ (see Eq.~\eqref{math:admm_stepy}) and send 
value $S^{d(B)}_i$ to load agent $i$, 
value $S^{g(B)}_i$ to generator agent $i$, and
values $S^{(B)}_{ij}, V^{(B)}_{i}$ to line agent $(ij)$, for each line ($i,j$) connected to bus $i$.

\item 
Finally, each agent updates its corresponding dual variables: 
$\lambda^{d}_i$, for load agent $i$,
$\lambda^g_i$, for generator agent $i$, and
$\lambda^{V}_{ij}, \lambda^{S}_{ij}$, for line agent $(i,j)$. 
Collectively, they are identified with 
$\bm{\lambda} = \langle \lambda^{d}_i, \lambda^g_i, \lambda^{V}_{ij}, \lambda^{S}_{ij}\rangle$, using the notation in Eq.~\eqref{math:admm_lambda}. 
\end{enumerate}

\begin{figure}[!t]
\centering
\includegraphics[width=0.9\linewidth]{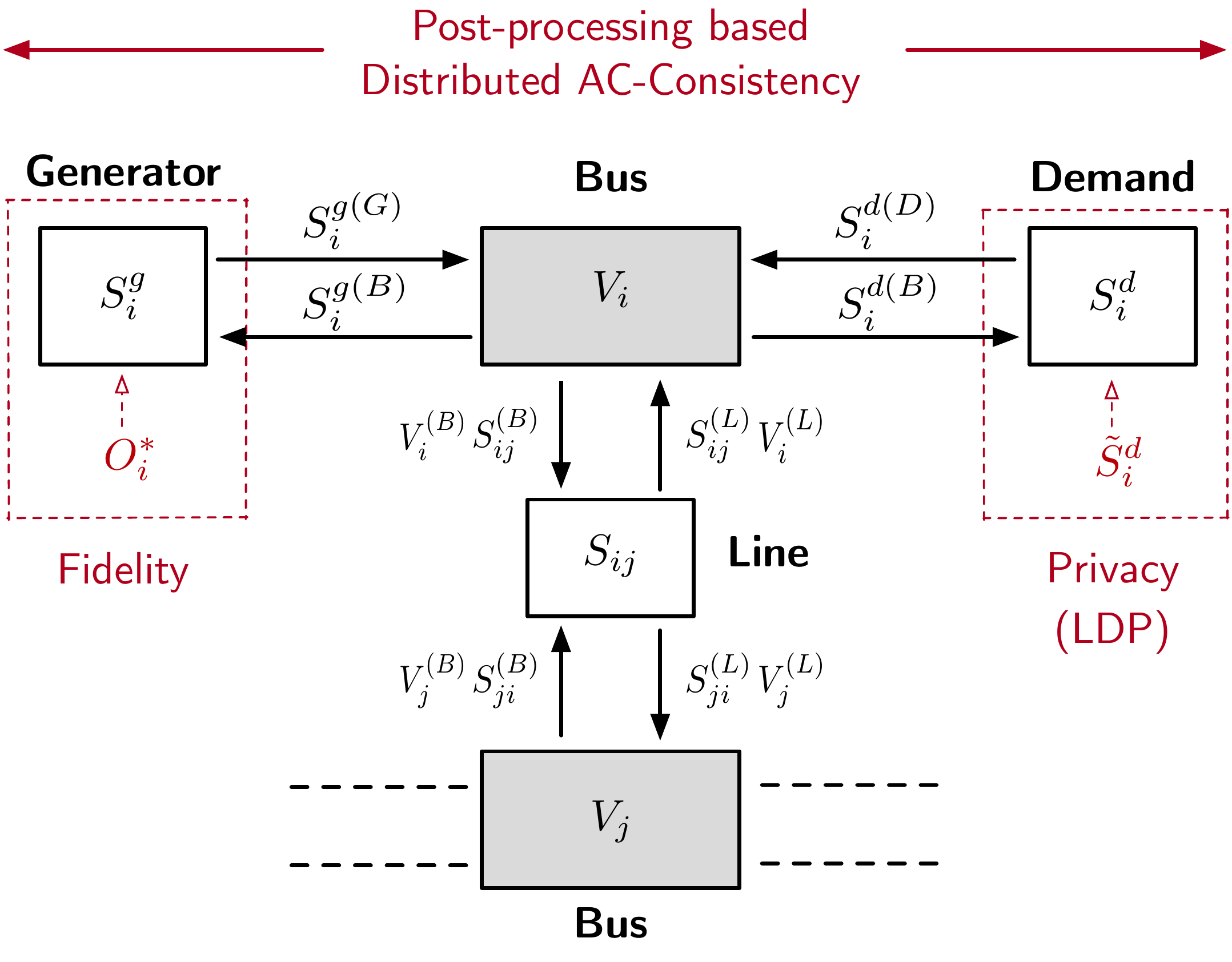}
\caption{The ADMM-based LDP post-processing step of PD-OPF.}
\label{fig:framework}
\end{figure}

The goal of the coupling constraints $A\bx + B\bz = \bc$ (see Constraint \eqref{math:admm}) is that of matching the states values $\bz$ of the bus agents to those of their connected components $\bx$.
The fidelity constraint (Eq.~\eqref{relax:fait}) and the AC Power Flow constraints (Eq.~\eqref{relax:set}) are enforced as local constraints by each agent.
Finally, the load agents control the minimization term Objective \eqref{obj:relax} of problem $P_\RBL$, to control the deviation of the new, post-processed load w.r.t.~the Laplace obfuscated counterpart. 
A detailed description of the individual optimization problems computing the local Lagrange functions (Eq.~\ref{math:admm_stepx}) for the load, generator, and line agents, and (Eq.~\ref{math:admm_stepy}) for the bus agents is given as follows.

\begin{model}[!t]
	{\footnotesize
	\caption{ADMM: Load agent $\cD_i (P_{load})$} 
	\label{model:admm_load}
	\vspace{-6pt}
	\begin{align}
		\mbox{\bf inputs:} \;\;
		&\langle \rho, \lambda^d_i, \tilde{S}^d_i, S^{d(B)}_i \rangle\nonumber \\
		\mbox{\bf variables:} \;\;
		& S^{d(D)}_i \nonumber \\
		\mbox{\bf minimize:} \;\;
		&  \| S^{d(D)}_i -  \tilde{S}^d_i \|^2 + \lambda^d_i \cdot S^{d(D)}_i + \frac{\rho}{2} \| S^{d(D)}_i -  S^{d(B)}_i \|^2 \label{admm_load_obj} 
	\end{align}
	}
	\vspace{-12pt}
\end{model}

\smallskip\noindent\textbf{Load agent}
The optimization step performed by each load agent $i$ ($i \in N$), at each iteration, produces a load value $S_i^{d(D)}$ and is shown in Model~\ref{model:admm_load}.
Eq.~\eqref{admm_load_obj} captures the load augmented Lagrange function (see Eq.~\eqref{obj:relax}) and the agent coupling constraints described as penalty terms. 
The first term of the objective is the L2 distance between the load value $S_i^{d(D)}$ and the Laplace load value $\tilde{S}^d_i$.
The remaining terms correspond to the load coupling constraint, matching the load values $S^{d(D)}_i$ to the feedback signal $S^{d(B)}_i$ from the connecting bus.

\begin{model}[!t]
	{\footnotesize
	\caption{ADMM: Generator agent $\cG_i (P_{gen})$}
	\label{model:admm_gen}
	\vspace{-6pt}
	\begin{align}
		\mbox{\bf inputs:} \;\;
		&\langle \rho, \lambda^g_i, \cO^*_i, S^{g(B)}_i \rangle \nonumber \\
		\mbox{\bf variables:} \;\;
		& S^{g(G)}_i \nonumber \\
		\mbox{\bf minimize:} \;\;
		&  \lambda^g_i \cdot S^{g(G)}_i + \frac{\rho}{2} \| S^{g(G)}_i -  S^{g(B)}_i \|^2 \label{admm_gen_obj} \\
		\mbox{\bf local constraints:} \;\;
		&  {S}^{gl}_i \leq S^{g(G)}_i \leq {S}^{gu}_i  \label{admm_gen_bound} \\
          	&  \cO^*_i (1 - \beta) \leq O(S^{g(G)}_i) \leq \cO^*_i (1 + \beta)  \label{admm_gen_costbound} 
	\end{align}
	}
	\vspace{-12pt}
\end{model}

\smallskip\noindent\textbf{Generator agent}
The objective of the generator agent $i$ ($i \in N$), at each iteration, is that of producing a dispatch value $S^{g(G)}_i$ that matches the feedback signal $S^{g(B)}_i$ from the connecting bus. 
The problem is reported in Model~\ref{model:admm_gen}. Therein, Eq.~\eqref{admm_gen_obj} describes the generator agent coupling constraints as penalty terms. The optimization model ensures that the dispatch values satisfy the feasible bounds (Eq.~\eqref{admm_gen_bound}), and that the dispatch cost stays within the fidelity requirement (Eq.~\eqref{admm_gen_costbound}).

\begin{model}[!t]
	{\footnotesize
	\caption{ADMM: Line agent $\cL_{i,j}  (P_{line})$}
	\label{model:admm_line}
	\vspace{-6pt}
	\begin{align}
		\mbox{\bf inputs:} \;\;
		&\langle \rho, \lambda^S_{ij}, \lambda^V_{ij}, S^{(B)}_{ij}, V^{(B)}_{ij},  
		 \lambda^S_{ji}, \lambda^V_{ji}, S^{(B)}_{ji}, V^{(B)}_{ji} \rangle \nonumber \\
%		    \lambda^S_{ji}, \lambda^V_{ji}, S^{T}_{ji}, V^{T}_{ji}  \rho \nonumber \\
		\mbox{\bf variables:} \;\;
		& S^{(L)}_{ij}, S^{(L)}_{ji}, V^{(L)}_{ij}, V^{(L)}_{ji} \nonumber \\
		\mbox{\bf minimize:} \;\;
		&  \sum_{(e,f) \in \{(i,j),(j,i)\}} 
		[\lambda^S_{ef} \cdot S^{(L)}_{ef} +  \lambda^V_{ef} \cdot V^{(L)}_{ef} + \nonumber \\
		&\qquad \qquad \frac{\rho}{2} ( \| S^{(L)}_{ef} -  S^{(B)}_{ef} \|^2  + 
		 \| V^{(L)}_{ef} -  V^{(B)}_{ef} \|^2 )] \label{admm_branch_obj} 
		\\
		\mbox{\bf local constr.} \;\;
		& \angle V^{(L)}_{ij} = 0, \;\; \mbox{if } i = s;   \angle V^{(L)}_{ji} = 0, \;\; \mbox{if } j = s;   \label{admm_branch_slack}\\
		\mbox{\bf} \;\;& {v}^{l}_e \leq |V^{(L)}_{ef}| \leq {v}^{u}_e, 		\;\; \forall (e,f) \in \{(i,j),(j,i)\}  \label{admm_branch_volt}\\
		& -{\theta}^\Delta_{ef} \leq \angle (V^{(L)}_{ef} V^{(L)*}_{fe}) \leq {\theta}^\Delta_{ef} \label{admm_branch_ang} \\
		& |S^{(L)}_{ef}| \leq {s}^u_{ef}, 			\;\; \forall (e,f) \in \{(i,j), (j,i)\} \label{admm_branch_thermal} \\
		& S^{(L)}_{ef} = {Y}^*_{ef} |V^{(L)}_{ef}|^2 - {Y}^*_{ef} V^{(L)}_{ef} V^{(L)*}_{ef} \nonumber\\
		& \qquad \qquad \qquad \qquad \qquad \forall (e,f)\in \{(i,j),(j,i)\} \label{admm_branch_acpf}
	\end{align}
	}
	\vspace{-12pt}
\end{model}

\smallskip\noindent\textbf{Line agent}
The objective of the line agent ($ij$) ($(i,j) \in E$), is that of finding flow values $S^{(L)}_{ij}$ and $S^{(L)}_{ji}$
and voltage values $V^{(L)}_{ij}$ and $V^{(L)}_{ji}$
that match the corresponding feedback signals $S^{(B)}_{ij}$, $V^{(B)}_{i}$ and $S^{(B)}_{ji}$, $V^{(B)}_{j}$, computed by the buses $i$ and $j$, respectively. 
The optimization is illustrated in Model~\ref{model:admm_line}. 
It describes four coupling constraints: two associated to the voltage values and two associated to the flow values (Eq.~\eqref{admm_branch_obj}). 
The model also ensures the voltages and power flows are within the feasible bounds  (Eqs.~\eqref{admm_branch_volt} to \eqref{admm_branch_thermal}),
and that the AC power flow constraints are satisfied (Eq.~\eqref{admm_branch_acpf}).
The voltage angle $ \angle V^{(L)}_{ij} /  \angle V^{(L)}_{ji} $ is zero if it connects to a slack bus (Eq.~\eqref{admm_branch_slack}).

\begin{model}[!t]
	{\footnotesize
	\caption{ADMM: Bus agent $\cB_i (P_{bus})$}
	\label{model:admm_bus}
	\vspace{-6pt}
	\begin{align}
		\mbox{\bf inputs:} \;\;
		& \langle \rho, \;\;
		   \lambda^d_i, S^{d(D)}_i, 
          	   \lambda^g_i, S^{g(G)}_i \rangle,
		   \nonumber \\
		& \langle \lambda^S_{ij}, S^{(L)}_{ij},  \lambda^V_{ij}, V^{(L)}_{ij} |  \forall (i,j) \in E \cup E^R  \rangle \nonumber \\
		%\lambda^S_{ji}, S^{(L)}_{ji},  \lambda^V_{ji}, V^{(L)}_{ji}  \rangle \nonumber \\
		\mbox{\bf variables:} \;\;
		& S^{d(B)}_i, S^{g(B)}_i, V^{(B)}_i, S^{(B)}_{ij} \forall (i,j) \in E \cup E^R   \nonumber \\
		\mbox{\bf minimize:} \;\;
		&  \lambda^d_i \cdot S^{d(B)}_i + \frac{\rho}{2} \| S^{d(B)}_i -  S^{d(D)}_i \|^2 + \nonumber\\
		&    \lambda^g_i \cdot S^{g(B)}_i + \frac{\rho}{2} \| S^{g(B)}_i -  S^{g(G)}_i \|^2 + \nonumber \\
		&  \sum_{(i,j) \in E \cup E^R } 
		[\lambda^S_{ij} \cdot S^{(B)}_{ij} + \frac{\rho}{2}  \| S^{(B)}_{ij} -  S^{(L)}_{ij} \|^2   + \nonumber \\
		&  \lambda^V_{ij} \cdot V^{(B)}_{i} + \frac{\rho}{2} \| V^{(B)}_{i} -  V^{(L)}_{ij} \|^2 ] \label{admm_branch_obj} 
		\\
		\mbox{\bf local constraint:} \;\;
		& 	 S^{g(B)}_i - {S}^{d(B)}_i = \sum_{(i,j)\in E \cup E^R} S^{(B)}_{ij}    \label{admm_bus_flow_balance}
	\end{align}
	}
	\vspace{-8pt}
\end{model}

\smallskip\noindent\textbf{Bus agent}
At each iteration, bus agent $i$ performs the optimization described in Model~\ref{model:admm_bus}.
Its objective is that of finding load value $S_i^{d(B)}$,
generator value $S_i^{g(B)}$,
voltage value $V_i^{(B)}$, 
and flow values $S_{ij}^{(B)}$, for each connecting line $(i,j) \in E \cup E^R$, that match the state variables sent from the load, generator, and line agents, respectively. 
The model also ensures the satisfaction of the flow balance constraint (Eq.~\eqref{admm_bus_flow_balance}).

%\smallskip
The ADMM coordination process coordinating 
all agents is illustrated in the Appendix (Algorithm~\ref{alg:admm}). 

Even though the ADMM agent structure comes from~\cite{mhanna19adaptive},
the ADMM scheme used by PD-OPF serves as a distributed resolution of $P_\RBL$, rather than 
a traditional scheme for solving OPF. 
It redistributes the noise introduced by the Laplace mechanism optimally to satisfy the fidelity criteria. 

\begin{theorem}
\label{th:dist_theorem1}
	PD-OPF satisfies local $\alpha$-indistinguishability.
\end{theorem}
\begin{proof}
By Theorem \ref{th:m_lap}, 
the load values obtained by the application of the Laplace mechanism 
satisfy $\alpha$-local indistinguishability. 
The ADMM mechanism makes use of exclusively the privacy-preserved load values $\bm{\tilde{S}}^d$ (computed by the application of the Laplace mechanism), as well as additional public information (e.g. the local cost values $\cO^*_i$).
Therefore, by post-processing immunity of differential privacy, 
PD-OPF satisfies local $\alpha$-indistinguishability.
\end{proof}

\section{Experimental Results}

This section reports on the obfuscation quality and ability to converge of PD-OPF. Additionally, the proposed method is compared with a centralized version that solves problem $P_{\RBL}$, thus admitting the presence of a centralized data curator. 
The experiments are performed on a variety of NESTA \cite{Coffrin14Nesta} benchmarks. 
Parameter $\epsilon$ is fixed to 1.0, the \emph{indistinguishability level} $\alpha$ varies from 0.01 to 0.1 in p.u. (i.e. 1 MVA to 10 MVA), and the \emph{fidelity level} $\beta$ varies from 10$^{-2}$ to 10$^{-1}$ (i.e. from 1\% to 10\% of the optimal cost difference). 
PD-OPF is limited to use 5000 iterations. 
All the models are implemented using PowerModels.jl~\cite{Coffrin:18} in Julia with nonlinear solver IPOPT~\cite{wachter06on}.

Choosing a fixed penalty factor $\rho$ to drive convergence is challenging~\cite{mhanna19adaptive}. 
Thus, the experimental routine adjusts $\rho$ dynamically, using the maximum primal and dual infeasibility values, $\epsilon_{p} = \max \bm{r}_p$ and $\epsilon_{d} = \max \bm{r}_d$, respectively (in spirit of~\cite{mhanna19adaptive}).
Higher values of $\rho$ encourage the satisfaction of the primal constraints, while lower values shift weights to the 
objectives and reduce the dual infeasibilities~\cite{mhanna19adaptive}. 
The heuristic adopted changes $\rho$ when the distance between $\epsilon_{p}$ and $\epsilon_{d}$ becomes too large:
\begin{align}
\rho &= \min\{(1+c)\rho, \overline{\rho} \}, &\mbox{ if } \epsilon_{p}  > c_t \epsilon_{d}, \nonumber \\
\rho &= \max\{ \frac{\rho}{1+c}, \underline{\rho} \}, &\mbox{ if } \epsilon_{d}  > c_t \epsilon_{p}. \nonumber 
\end{align}
The scaling factor $c$ is set to $2\%$, the threshold parameter $c_t$ to $7.0$, and upper $\overline{\rho}$ and lower $\underline{\rho}$ bounds to $10^{6}$ and $5$, respectively.

To allow PD-POPF to restore primal feasibility, a feasibility boosting procedure is implemented as follows. 
When the iteration counter 4500 iterations, if the maximum primal infeasibility is larger than $10^{-3}$, $\rho$ will be increased by: $\min\{(1+c)\rho, \overline{\rho} \}$.
We call this phase \emph{feasibility boosting}.

\begin{figure}[!tb]
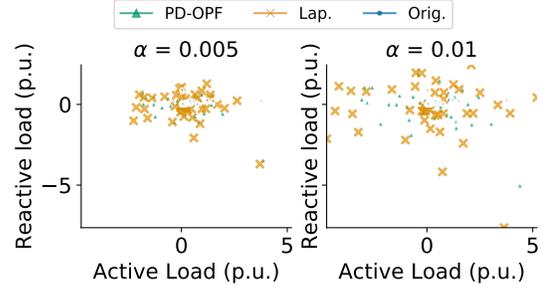

	\centering
	\includegraphics[width=.25\textwidth,valign=t]{results/load-scatter-plot/{{Legend}}}\\
 	\includegraphics[width=.40\textwidth,valign=t]{results/load-scatter-plot/{{nesta_case57_ieee_8_0.1}}}
	\caption{Original loads distance from the Laplace and the PD-OPF Mechanisms on the IEEE-57 bus system, at varying of the indistinguishability value $\alpha = 0.005$ (left) and $\alpha=0.01$ (right).}
	\label{fig:load}
\end{figure}

\subsection{Quality of Load Demand Obfuscation}

Figure~\ref{fig:load} depicts the original load values (Orig.) associated to the IEEE-57 bus systems, and compares them with those generated by the Laplace mechanism (Lap.) and by PD-OPF. 
The figure illustrates that the post-processing step used in PD-OPF modifies the original loads.
Since the Laplace mechanism does not converge to an AC feasible solution, PD-OPF further modifies the Laplace-generated loads. 
The figure does not report the AC-feasible loads due to large overlaps with PD-OPF values.

\begin{figure*}[!t]
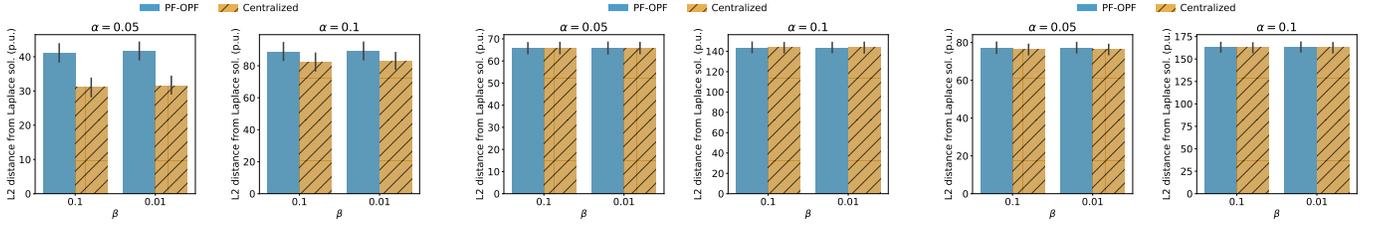

	\centering
	\includegraphics[width=.31\textwidth,valign=c]{results/load-distance-to-laplace-plot/{{nesta_case39_epri_A}}}
	\hspace{.02\textwidth}
	\includegraphics[width=.31\textwidth,valign=c]{results/load-distance-to-laplace-plot/{{nesta_case57_ieee_A}}}
	\hspace{.02\textwidth}
	\includegraphics[width=.31\textwidth,valign=c]{results/load-distance-to-laplace-plot/{{nesta_case189_edin_A}}}

	\caption{L2 distance between the Laplace obfuscated data and the PD-OPF and the Centralized obfuscated data. 
	IEEE-39 (left), IEEE-57 (center), IEEE-189 (right). $\alpha = [0.05, 0.1]$, $\beta=[0.01,0.1]$.} 
	\label{fig:lap-dist}
\end{figure*}

\subsection{Quality of Privacy Loss Minimization}

Figure~\ref{fig:lap-dist} illustrates the difference between the loads produces by PD-OPF and those produces by a centralized implementation of problem $P_\RBL$~\cite{Fioretto:18b}. 
The difference is measured in terms of distance from the Laplace obfuscated loads (averaged over 50 instances).
The differences in the IEEE-39 test case are due to the feasibility boosting phase, activated to improve the primal feasibility. 
In the other test cases the differences between the two approaches are negligible, thus validating the use of a decentralized solution for releasing loads when a centralized trusted data curator is unavailable. 
 
\subsection{Quality of Fidelity Restoration}

%Figure~\ref{fig:opf} illustrates the average difference (in percentage) in dispatch cost computed from the obfuscated PD-OPF data vs the optimal one

Figure~\ref{fig:opf} illustrates the average percentage difference on the dispatch cost differences $\hat{\cO}$ between the original and obfuscated loads produced by PD-OPF: 
$100 \times \frac{\cO( \hat{\cO} - \cO(P_{\text{OPF}}({\bm{S}}^d)) }{ \cO(P_{\text{OPF}}({\bm{S}}^d)) }$. 
Since a PD-OPF implements a relaxation of Constraint \eqref{c:oli_set}, the Figure also reports a comparison using a centralized procedure that solves an AC-OPF with the PD-OPF loads as input. 
%\footnote{
%Since the AC-OPF validation routine often require a higher level of numerical accuracy, the closest AC feasible loads (in L2 norm) at required numerical accuracy will be computed and used instead.}
The experimental results indicate that PD-OPF is able to restore the problem fidelity well, even when the fidelity requirement $\beta$ are as small as $0.01$\% of the original costs.

\begin{figure}[h]
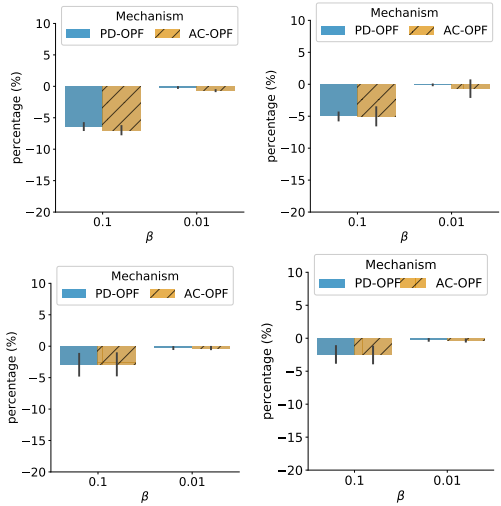

	\centering
 	\includegraphics[width=.18\textwidth]{results/opf-diff/{{nesta_case39_epri_A0.01}}}
 	\includegraphics[width=.18\textwidth]{results/opf-diff/{{nesta_case39_epri_A0.1}}}\\
 	\includegraphics[width=.18\textwidth]{results/opf-diff/{{nesta_case57_ieee_A0.01}}}
 	\includegraphics[width=.18\textwidth]{results/opf-diff/{{nesta_case57_ieee_A0.1}}}
	\caption{Dispatch costs differences between the optimal and the PD-OPF solution (PD-OPF) and its centralized AC-OPF counterpart. 
	IEEE 39 (top) \& IEEE 57 bus (bottom), $\alpha = 0.01$ (left), $0.1$ (right), $\beta=[0.1,0.01]$.} 
	\label{fig:opf}
\end{figure}

\subsection{Convergence Quality \& Runtime}
Finally, table~\ref{tbl:conv1} presents the maximum and dual infeasibilities (in p.u.), before and after (marked with $*$) activating the feasibility boosting procedure. 
The table clearly illustrates the benefits of the boosting procedure.
It is able to reduce the primal infeasibility of up to two order of magnitude, albeit at a cost of a larger dual infeasibility. 

Figure~\ref{fig:conv_39} illustrates the details of one run on the IEEE-39 benchmark. 
After a few iterations, both the primal and the dual infeasibilities stabilize in the range $[10^1, 10^{-1}]$ (top-left), 
and the generator costs stabilize after 2000 iterations (bottom-left).
When the feasibility boosting is activated, the coordination agent increases the parameter $\rho$ (bottom right), inducing all agents to re-optimize with a higher penalty for violating the coupling constraints. This is obtained at a cost of a larger dual feasibility (top-right).

\begin{table}[h]
\centering
\setlength{\tabcolsep}{5pt}
\caption{Primal \& dual infeasibility, and Sim. runtime. $\alpha = 0.1, \beta = 0.1$. }
\resizebox{0.95\columnwidth}{!}{
\begin{tabular}{|l|rr|rr|r|}
\toprule
{} &  Primal &  Primal$*$ &  Dual &  Dual$*$ &   Time (min.) \\
\midrule
nesta\_case3\_lmbd      &        0.036 &                0.001 &      0.173 &              0.079 &   1.147 \\
nesta\_case4\_gs        &        0.023 &                0.001 &      0.092 &             13.953 &   2.110 \\
nesta\_case5\_pjm       &        1.580 &                0.015 &      3.094 &            380.243 &   3.501 \\
nesta\_case6\_c         &        0.203 &                0.001 &      0.835 &              7.088 &   2.607 \\
nesta\_case6\_ww        &        0.094 &                0.001 &      0.419 &              7.919 &   3.215 \\
nesta\_case9\_wscc      &        0.197 &                0.001 &      1.224 &              5.908 &   2.776 \\
nesta\_case14\_ieee     &        0.579 &                0.001 &      2.228 &             19.762 &   5.141 \\
nesta\_case24\_ieee\_rts &        0.293 &                0.006 &      1.276 &            540.403 &  11.157 \\
nesta\_case29\_edin     &        0.216 &                0.128 &      2.393 &           3027.724 &  38.686 \\
nesta\_case30\_as       &        0.386 &                0.001 &      1.685 &             15.223 &  12.592 \\
nesta\_case30\_fsr      &        0.416 &                0.001 &      2.215 &             14.001 &  10.738 \\
nesta\_case30\_ieee     &        0.621 &                0.001 &      2.831 &             37.137 &  11.167 \\
nesta\_case39\_epri     &        0.291 &                0.026 &      1.597 &           1849.358 &  15.273 \\
nesta\_case57\_ieee     &        0.584 &                0.001 &      2.951 &             62.776 &  19.614 \\
nesta\_case73\_ieee\_rts &        0.402 &                0.008 &      2.762 &            691.576 &  45.131 \\
nesta\_case118\_ieee    &        0.968 &                0.004 &      4.427 &            394.885 &  82.160 \\
nesta\_case189\_edin    &        3.214 &                0.017 &     14.780 &            871.245 &  86.908 \\
\bottomrule
\end{tabular} \label{tbl:conv1}
}
\end{table}

\begin{figure}[h]
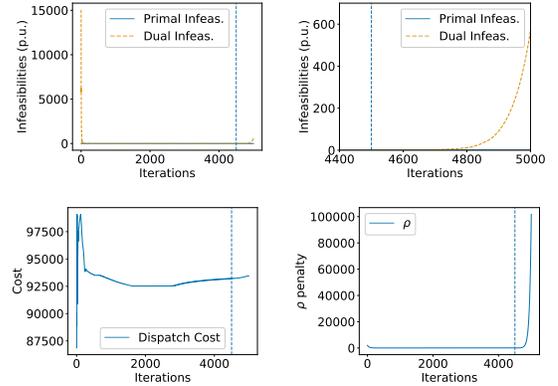

	\centering
 	\includegraphics[width=.20\textwidth,valign=c]{results/iter-data/{{nesta_case39_epri_1_0.1-infeas}}}
 	\includegraphics[width=.20\textwidth,valign=c]{results/iter-data/{{nesta_case39_epri_1_0.1-infeaszoom}}}
	\includegraphics[width=.20\textwidth,valign=c]{results/iter-data/{{nesta_case39_epri_1_0.1-cost}}}
	\includegraphics[width=.20\textwidth,valign=c]{results/iter-data/{{nesta_case39_epri_1_0.1-rho}}}	
	\caption{IEEE-39 bus: Primal $\epsilon_{p}$ and dual $\epsilon_{d}$ infeasibilities (Full-scale: top left, $>4400$ iterations: top right), generator dispatch costs (bottom left), penalty $\rho$ (bottom right); $\alpha =0.1$, $\beta=0.1$.
	The vertical dotted line marks the activation of the boosting procedure.} 
	\label{fig:conv_39}
\end{figure}

\section{Conclusion}
This paper presents a distributed framework based on ADMM
and Local Differential Privacy (LDP)
to preserve the privacy of customer loads while maintaining 
power flows close to the optimal solution.
We formally present the distributed privacy-preserving problem,
and a two phase distributed mechanism Privacy-preserving Distributed
OPF (PD-OPF) to guarantee privacy and fidelity.
The mechanism satisfies privacy properties.
Experimental evaluations on the NESTA benchmarks 
show that the mechanism provides high obfuscation quality,
satisfies the fidelity requirements, and achieves
comparable results when compared to a centralized approach. 

\section*{Acknowledgement}
The authors would like to thank Kory Hedman for extensive discussions on various obfuscation techniques. This research is partly funded by the ARPA-E Grid Data Program under Grant 1357-1530.

\newpage

\bibliographystyle{IEEEtran}
\bibliography{differential_privacy,opf_bib}

\newpage

\appendix

\input{appendix}

% that's all Pascal & Nando!
\end{document}

%% file: nando_commands.tex
\usepackage{latexsym}
\usepackage{url}
\ifCLASSINFOpdf
  \usepackage[pdftex]{graphicx}
\else
  \usepackage[dvips]{graphicx}
\fi
\usepackage{picinpar}
\usepackage{color, soul, colortbl}
\usepackage{wrapfig}
\usepackage{mdframed}
\usepackage{amsmath,amsfonts,mathrsfs,amssymb,bm,mathtools,mathabx}
\usepackage{xspace}
\usepackage[pdftex,dvipsnames]{xcolor}  % Coloured text etc.
\usepackage{todonotes}
\usepackage[ruled,linesnumbered,noresetcount,vlined]{algorithm2e}
\usepackage{hyperref} 
\hypersetup{
colorlinks=true, 
breaklinks=true,
urlcolor= blue, 
linkcolor= blue,
citecolor=red, 
}
\usepackage{float}
\floatstyle{ruled}
\newfloat{model}{thp}{lop}
\floatname{model}{Model}

\usepackage{multicol}

\usepackage{booktabs}
\usepackage{longtable}
\usepackage{xcolor}

\usepackage{enumitem}

\newcommand{\bitemize}{\begin{list}{$\bullet$}{\topsep=1pt \parsep=0pt \itemsep=1pt \leftmargin=1em }} 
\newcommand{\eitemize}{\end{list}}
\newcommand{\benumerate}{\hspace{-0.5in}\begin{enumerate}\topsep=0pt \parsep=0pt \itemsep=-3pt} 
\newcommand{\eenumerate}{\end{enumerate}}
\newcommand{\beitemize}{\begin{list}{$\bullet$}{\topsep=1.5pt \parsep=0pt \itemsep=1pt \leftmargin=1em }} 
\newcommand{\enitemize}{\end{list}}

\def\thmspace{0.2em}
\newtheorem{theorem}{\hspace{\thmspace}{\bf Theorem}\!}

\newenvironment{proof}{{\textit{Proof}.}}{\hfill$\Box$}

\DeclareMathOperator*{\argmin}{argmin}

\newcommand{\cB}{\mathcal{B}}

\newcommand{\cD}{\mathcal{D}}

\newcommand{\cG}{\mathcal{G}}

\newcommand{\cL}{\mathcal{L}}
\newcommand{\cM}{\mathcal{M}}
\newcommand{\cN}{\mathcal{N}}
\newcommand{\cO}{\mathcal{O}}

\newcommand{\cZ}{\mathcal{Z}}
\newcommand{\cX}{\mathcal{X}}

\newcommand{\bS}{\bm{S}}

\newcommand{\bc}{\bm{c}}
\newcommand{\bx}{\bm{x}}

\newcommand{\bz}{\bm{z}}

\newcommand{\RR}{\mathbb{R}}

\newcommand{\CC}{{\mathbb{C}}}
\newcommand{\sC}{{\mathscr{C}}}

\newcommand{\Lap}{{\text{Lap}}}

%% file: appendix.tex
\subsection{PD-OPF with the Piecewise Mechanism}

The PD-OPF framework can also be extended to work with other 
Local Differential Privacy mechanism (LDP).
Instead of using Polar Laplace mechanism in the Privacy Phase,
this section showcases another LDP mechanism: the Piecewise Mechanism \cite{wang2019collecting}.
The Piecewise Mechanism also satisfies the $\bm{L}^1\epsilon$-LDP for $\alpha$ distances definition.
It requires all input data $x_i$ to be normalized within $[-1, 1]$ from $[\underline{x_i}, \overline{x_i}]$.
Let:
\begin{align*}
	C &= \frac{e^{\epsilon/2\alpha}+1}{e^{\epsilon/2\alpha}-1}, \\
	L(x_i) &= \frac{C+1}{2} x_i - \frac{C-1}{2}, \mbox{ and}\\
	R(x_i) &= L(x_i) + C-1.
\end{align*}
The mechanism perform obfuscation based as in Algorithm \ref{alg:piecewise}. 
\begin{algorithm}
\caption{Piecewise Mechanism for LDP}
\label{alg:piecewise}
	Sample $p \sim \text{Uniform}([0,1])$\\
	\uIf{$p \leq \frac{e^{\epsilon/2\alpha}}{e^{\epsilon/2\alpha}+1}$}
	 {
	 Sample $\tilde{x}_i \sim \text{Uniform}([L(x_i), R(x_i)])$
	 }
	\Else{ 
	Sample $\tilde{x}_i \sim \text{Uniform}([-C, L(x_i)] \cup [R(x_i), C])$}
	Return $\tilde{x}_i$
\end{algorithm}
To implement the Piecewise Mechanism, linear transformations are used by each of the load agents $\cD_i$ to normalize active and reactive parts of 
the load value $S^d_i$ into [-1,1].
To transform from a bounded domain $x_i \in [\underline{x_i}, \overline{x_i}]$ to $y_i \in [-1, 1]$ (and vice versa), the following equation is used:
$y_i = 2 \frac{x_i - \underline{x_i}}{\overline{x_i} - \underline{x_i}} - 1$.

Table~\ref{tbl:piecewise_conv} shows the primal and dual convergence quality and simulation
runtime similar as in previous section. Figure~\ref{fig:piecewise_opf_57}
shows the fidelity can again be restored by the ADMM mechanism.
Figure~\ref{fig:piecewise_load} shows comparable obfuscation quality when comparing to the 
Laplace mechanism in Figure~\ref{fig:load}.
Finally, Figure~\ref{fig:piecewise-lap-dist} 
shows the ADMM algorithm can achieve comparable privacy
loss minimization results to centralized optimization. 

%Similarly, to transform from $y_i \in [-1, 1]$ back to $x_i \in [\underline{x_i}, \overline{x_i}]$:
%\begin{align}
%x_i = (\overline{x_i} - \underline{x_i}) \frac{y_i + 1}{2} +  \underline{x_i}
%\end{align}

\begin{table}[h]
\centering
\setlength{\tabcolsep}{5pt}
\caption{Primal and dual infeasibility, and simulation runtime. $\alpha = 0.1, \beta = 0.1$. }
\resizebox{\columnwidth}{!}{
\begin{tabular}{|l|rr|rr|r|}
\toprule
{} &  Primal &  Primal$*$ &  Dual &  Dual$*$ &   Time (min.) \\
\midrule
nesta\_case3\_lmbd      &        0.001 &                0.001 &      0.015 &              0.015 &   0.089 \\
nesta\_case4\_gs        &        0.031 &                0.001 &      0.151 &             11.733 &   3.505 \\
nesta\_case5\_pjm       &        1.820 &                0.015 &      3.290 &            382.929 &   3.416 \\
nesta\_case6\_c         &        0.006 &                0.001 &      0.038 &              0.180 &   0.479 \\
nesta\_case6\_ww        &        0.217 &                0.072 &      1.064 &          20667.869 &   4.165 \\
nesta\_case9\_wscc      &        0.023 &                0.001 &      0.119 &              1.596 &   1.445 \\
nesta\_case14\_ieee     &        0.085 &                0.001 &      0.392 &              5.402 &   8.722 \\
nesta\_case24\_ieee\_rts &        0.133 &                0.008 &      0.859 &            611.856 &  11.294 \\
nesta\_case29\_edin     &        0.197 &                0.098 &      2.676 &           3810.460 &  82.134 \\
nesta\_case30\_as       &        0.161 &                0.001 &      0.847 &              5.090 &   9.244 \\
nesta\_case30\_fsr      &        0.050 &                0.001 &      0.250 &              1.525 &   9.824 \\
nesta\_case30\_ieee     &        0.211 &                0.001 &      1.074 &              9.788 &   9.714 \\
nesta\_case39\_epri     &        0.920 &                0.020 &      4.371 &           1029.756 &  37.142 \\
nesta\_case57\_ieee     &        1.201 &                0.001 &      4.947 &            104.336 &  40.772 \\
nesta\_case73\_ieee\_rts &        0.219 &                0.011 &      1.654 &            777.139 &  43.815 \\
nesta\_case189\_edin    &        1.432 &                0.016 &      6.904 &            799.463 &  83.319 \\
\bottomrule
\end{tabular} \label{tbl:piecewise_conv}
}
\end{table}

\begin{figure}[h]
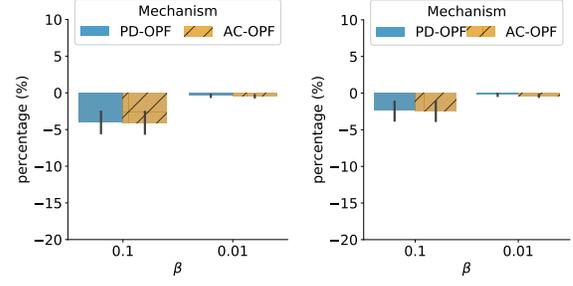

	\centering
 	\includegraphics[width=.21\textwidth]{results/piecewise/opf-diff/{{nesta_case57_ieee_A0.01}}}
 	\includegraphics[width=.20\textwidth]{results/piecewise/opf-diff/{{nesta_case57_ieee_A0.1}}}
	\caption{IEEE 57 bus. Percentage Difference on the Dispatch Costs after ADMM mechanism and AC validation: $\alpha = 0.01 (left), 0.1 (right)$, $\beta=[0.1,0.01]$. Average over 50 instances.} 
	\label{fig:piecewise_opf_57}
\end{figure}

\begin{figure}[h]
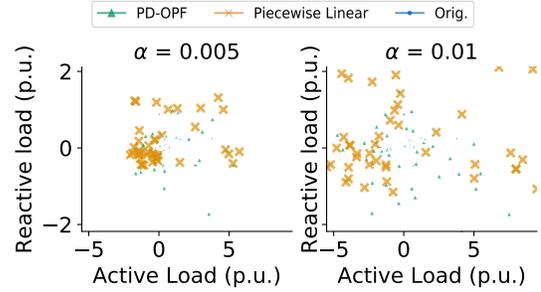

	\centering
	\includegraphics[width=.28\textwidth,valign=t]{results/piecewise/load-scatter-plot/{{Legend}}}
 	\includegraphics[width=.40\textwidth,valign=t]{results/piecewise/load-scatter-plot/{{nesta_case57_ieee_8_0.1}}}
	\caption{Loads from the original dataset and the Piecewise Linear \& ADMM Mechanisms on the IEEE-57 bus system.}% $\alpha = [0.01, 0.005]$}%, $\beta=0.1$.} 
	\label{fig:piecewise_load}
\end{figure}

\begin{figure}[h]
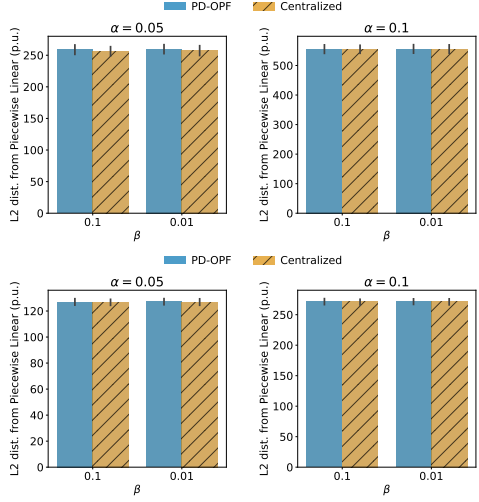

	\centering
	\includegraphics[width=.35\textwidth,valign=c]{results/piecewise/load-distance-to-laplace-plot/{{nesta_case39_epri_A}}}
	~~~
	\includegraphics[width=.35\textwidth,valign=c]{results/piecewise/load-distance-to-laplace-plot/{{nesta_case57_ieee_A}}}

	\caption{L2 distance between ADMM/Centralized Mechanisms and Piecewise Linear Obfuscated Dataset on the IEEE-39 (top) and IEEE-57 (bottom) bus systems, with $\alpha = [0.05, 0.1]$, and $\beta=[0.01,0.1]$.} 
	\label{fig:piecewise-lap-dist}
\end{figure}

%When the input value $x_i$ is defined in the interval $[-w, w]$, with $w \neq 1$, one first computes $x_i' = x_i / w$, then perturbs $x_i'$ using Algorithm \ref{alg:1}, and finally releases $\tilde{x}_i' \cdot w$.

\subsection{ADMM Coordination Process}
The ADMM coordination process executed by all PD-OPF agents is illustrated in Algorithm~\ref{alg:admm}.
Lines 1 to 3 initialize all variables associated to load, generator, and line agents, respectively. 
These agents, hence, perform their optimization step (lines 6 to 8), independently, and send their state (consensus) variables to the corresponding bus agents.
Next, the bus agents perform the associated local optimization step and send the feedback values back to the corresponding load, generator, and line agents (line 10).
Finally, the multipliers variables $\lambda$ are updated by each individual agents (lines 12 to 14).
At the end of each iteration, the parameter $\rho$ can be updated bu all agents.

\begin{algorithm}[t]
\SetKwInOut{Input}{Inputs}
\SetKwInOut{Output}{Output}
\caption{ADMM: Main routine}
\label{alg:admm}
{\scriptsize

\SetKwInOut{Input}{Inputs}
\SetKwInOut{Output}{Output}
\SetKwFunction{PLoad}{$P_{load}$}
\SetKwFunction{PGen}{$P_{gen}$}
\SetKwFunction{PLine}{$P_{line}$}
\SetKwFunction{PBus}{$P_{bus}$}
\SetKwProg{Fn}{Function}{:}{}
\SetKwProg{Mn}{Algorithm}{:}{}

\Input{$\langle \bm{\cN}, \rho_{init}, t_{max} \rangle$, $\langle \tilde{S}^d_i | \forall \cD_i \rangle$, $\langle O^{*}_i | \forall \cG_i \rangle$}
$\rho \gets \rho_{init}$\\
$\langle \lambda^d_i, S^{d(B)}_i \rangle \gets \langle 0, 0 \rangle \;\; \forall \cD_i; \;\; $
$ \langle \lambda^g_i, S^{g(B)}_i \rangle \gets \langle 0,0 \rangle \;\; \forall \cG_i; $\\
$\langle \lambda^S_{ij}, S^{(B)}_{ij} \rangle \gets \langle 0,0 \rangle \wedge \langle \lambda^V_{ij}, V^{(B)}_{ij} \rangle \gets \langle 0,0 \rangle \;\; \forall (i,j) \in \cL_{i,j} $\\
\For{$t = 1, 2, \ldots, t_{max}$}
{
Optimization of load, generator, and line agents\\
$\forall \cD_i: S^{d(D)}_i \gets $  \PLoad{$\langle \rho, \lambda^d_i, \tilde{S}^d_i, S^{d(B)}_i \rangle $}\\
$\forall \cG_i: S^{g(G)}_i \gets$  \PGen{$\langle \rho, \lambda^g_i, \cO^*_i, S^{g(B)}_i \rangle$}\\
$\forall \cL_{i,j}: S^{(L)}_{ij}, V^{(L)}_{ij}, S^{(L)}_{ji}, V^{(L)}_{ji} \gets$ 
 \PLine{$\langle \rho, \lambda^S_{ij}, \lambda^V_{ij}, S^{(B)}_{ij}, V^{(B)}_{ij},  \lambda^S_{ji}, \lambda^V_{ji}, S^{(B)}_{ji}, V^{(B)}_{ji} \rangle   $}\\
Optimization of bus agents\\
$\forall \cB_i: 
S^{d(B)}_i, S^{g(B)}_i, V^{(B)}_i, S^{(B)}_{ef} \gets 
\PBus(  \langle \rho, -\lambda^d_i, S^{d(D)}_i,  -\lambda^g_i, S^{g(G)}_i \rangle,  \langle -\lambda^S_{ef}, S^{(L)}_{ef},  -\lambda^V_{ef}, V^{(L)}_{ef}  \rangle  )$\\
Lagrange multiplier update\\
$\forall \cD_i \mbox{ and } \cB_i: \lambda^{d}_i \gets \lambda^{d}_i + (S^{d(D)}_i - S^{d(B)}_i)$\\
$\forall \cG_i \mbox{ and } \cB_i: \lambda^{g}_i \gets \lambda^{g}_i + (S^{g(G)}_i - S^{g(B)}_i) $ \\
$\forall \cL_{i,j} \mbox{ and } \cB_i / \cB_j: \lambda^S_{ij} \gets \lambda^S_{ij} + (S^{(L)}_{ij} - S^{(B)}_{ij}),  \lambda^V_{ij} \gets \lambda^V_{ij} + (V^{(L)}_{ij} - V^{(B)}_{i})$\\
Coordinating agent penalty $\rho$ update (optional)\\
$\rho \gets $ update\_p()\\
}
%$\bm{\hat{S}}^d_{R}, \lambda^{l}, \lambda^{u}  \gets$  {\FR{$\langle c \rangle$}}\\
%$\bm{\hat{S}}^{d} \gets ${\FC{$\langle \lambda^{l}, \lambda^{u}, \bm{\hat{S}}^d_{R}, \lambda^{\textsl{tol}} \rangle$}}\\
\Output{$S^{d}_i$}
%\Fn{\FR{$\langle c \rangle$}}
%{
%$\lambda \gets 1.0, \lambda^l \gets 1.0, \lambda^u \gets \infty$\\
%%
%$\bm{\hat{S}}^d_{R} \gets 
%	P_{\text{R}} ()$ 
%	\\
%%%%%%%%%%
%\For{$i = 1, 2, \ldots$}
%{
%        $\bm{\hat{S}}^d_{(i)} \gets P_{\text{M}} (\lambda, \bm{\hat{S}}^d_{R})$\\
%	$\bm{S}^g_{(i)} \gets P_{\text{OPF}}(\bm{\hat{S}^d}_{(i)})$\\
%	%
%	\lIf{$|\cO(\bm{S}^g_{(i)}) - \cO^*| \leq \beta \cO^* $}  {$\lambda^u \gets \lambda, $ \textbf{terminate}}
%	$\lambda^l \gets \lambda$, $\lambda \gets c \lambda$
%}
%\KwRet $\bm{\hat{S}}^d_{R}, \lambda^{l}, \lambda^{u}$
%}
}
\end{algorithm}